\documentclass{amsart}
\usepackage{amssymb}
\usepackage{amsmath}
\usepackage{tikz-cd}
\usepackage{amsthm}
\theoremstyle{plain}

\theoremstyle{plain}
\newtheorem{thm}{Theorem}
\newtheorem{prop}{Proposition}
\newtheorem{lem}{Lemma}

\usepackage{geometry}
\theoremstyle{definition}
\newtheorem{exam}{Example}
\newtheorem{rem}{Remark}

\newtheorem{definition}{Definition}

\begin{document}
\setcounter{page}{1}

\title[A few remarks on bounded homomorphisms]{ A few remarks on bounded homomorphisms acting on topological lattice groups and topological rings }

\author[Omid Zabeti ]{Omid Zabeti}

\address{ Department of Mathematics, University of Sistan and Baluchestan, P.O. Box: 98135-674, Zahedan, Iran.}
\email{{o.zabeti@gmail.com}}



\date{Received: xxxxxx; Revised: yyyyyy; Accepted: zzzzzz.}

\begin{abstract}
Suppose $G$ is a locally solid lattice group. It is known that there are non-equivalent classes of bounded homomorphisms on $G$ which have topological structures. In this paper, our attempt is  to assign lattice structures on them. More precisely, we use of  a version of the remarkable Riesz-Kantorovich formulae  and Fatou property for bounded order bounded homomorphisms to allocate the desired structures. Moreover, we show that unbounded convergence on a locally solid lattice group is topological and we investigate some applications of it.
Also, some necessary and sufficient conditions for completeness of different types of bounded group homomorphisms between topological rings have been obtained, as well.
\end{abstract}
\maketitle





\section{Introduction and Preliminaries}
The concept of a lattice group ($\ell$-group, for short) was initially considered in \cite{B,C}. In addition, topological $\ell$-groups as an extension of topological Riesz spaces were investigated in \cite{S1,S2}. Since the most known classes of function spaces are Banach lattices: one of the most powerful tools in the theory of Banach spaces, and Riesz spaces are the fundamental basis of Banach lattices, these notions have been investigated extensively from the past until now. But topological $\ell$-groups are rarely utilized although in general, topological groups have many applications in other disciplines for example Fourier analysis. Recently, a suitable reference has been announced regarding basic properties of topological $\ell$-groups ( see \cite{H} for more details on these expositions).

On the other hand, in \cite{KZ}, Kocinac and the author, considered three different kinds of bounded homomorphisms on a topological group. They allocated each class of them to an appropriate topology and showed that they form again  topological groups. If the underlying group has a lattice structure ( for example topological $\ell$-groups), it is of interest to ask whether bounded homomorphisms can have a lattice construction, too?
This question for bounded order bounded operators on locally solid Riesz spaces have been answered affirmatively in \cite{EGZ}. Almost, the most fruitful structure for the lattice operations in order bounded operators is the remarkable Riesz-Kantorovich formulae ( see \cite[Theorem 1.18]{AB} for more information).
Thus, in prior to anything, for order bounded homomorphisms on topological $\ell$-groups, we need a version of this formulae; this is done recently in \cite{Z2}. Then, we can consider lattice structures for classes of bounded order bounded homomorphisms.
A related and major point to consider is that although some proofs in this paper might seem similar to the ones related to Riesz spaces at the first glance, It is obligatory to check them one by one because some known results in analysis rely heavily on scalar multiplication such as the Hahn-Banach theorem and some consequences of it; so that we can not expect them in topological $\ell$-groups. But order structure enables us to generalize some results in Riesz spaces which count on just group and order structures. Recently, among other things, some extensions of this kind, have been considered in \cite{Z2}.

We organize the paper as follows. First, we consider some preliminaries and terminology which will be used in the sequel. In Section 2, we investigate a method which enables us to allocate lattice structures on bounded homomorphisms between topological $\ell$-groups.
In fact, we use the Fatou property with a version of the Riesz-Kantorovich formulae to give a lattice structure to bounded order bounded homomorphisms. Also, we see that unbounded convergence in a locally solid $\ell$-group is topological and we state some points in this direction.

In Section 3, we show that each class of bounded group homomorphisms defined on a topological ring is topologically complete if and only if so is the underlying topological ring.

By a {\bf lattice group} ( $\ell$-group), we mean a group which is also a lattice at the same time. Observe that a subset $B$ in an abelian topological group $(G,+)$ is said to be {\bf bounded} if for each neighborhood $U$ of the identity, there exists a positive integer $n$ with $B\subseteq nU$, in which $nU=\{x_1+\ldots +x_n: x_i\in U\}$. An $\ell$-group $G$ is called {\bf Dedekind complete} if every non-empty bounded above subset of $G$ has a supremum. $G$ is {\bf Archimedean} if $nx\leq y$ for each $n\in \Bbb N$ implies that $x\leq 0$. One may verify easily that every Dedekind complete $\ell$-group is Archimedean. In this note, all groups are considered to be abelian. A set $S\subseteq G$ is called {\bf solid} if $x\in G$, $y\in S$ and $|x|\leq |y|$ imply that $x\in S$.

Note that by a {\bf topological lattice group}, we mean a topological group which is simultaneously a lattice whose lattice operations are also continuous with respect to the assumed topology.

Suppose $G$ is a topological $\ell$-group. A net $(x_{\alpha})\subseteq G$ is said to be {\bf order} convergent to $x\in G$ if there exists a net $(z_{\beta})$ ( possibly over a different index set) such that $z_{\beta}\downarrow 0$ and for every $\beta$, there is an $\alpha_0$ with $|x_{\alpha}-x|\leq z_{\beta}$ for each $\alpha\ge \alpha_0$. A set $A\subseteq G$ is called {\bf order closed} if it contains limits of all order convergent nets which lie in $A$.

 Keep in mind that topology $\tau$ on a topological $\ell$-group $(G,\tau)$ is referred to as {\bf Fatou} if it has a local basis at the identity consists of solid order closed neighborhoods.

For undefined expressions and the related topics, see \cite{H}.

Now, we recall some terminology we need in the sequel ( see \cite{KZ} for further notifications about these facts).

\begin{definition}\rm
Let $G$ and $H$ be topological groups. A homomorphism $T:G \to
H$ is said to be
\begin{itemize}
\item[$(1)$] \emph{{\sf nb}-bounded} if there exists a
neighborhood $U$ of $e_G$ such that $T(U)$ is bounded in $H$;

\item[$(2)$] \emph{{\sf bb}-bounded} if for every bounded set $B
\subseteq G$, $T(B)$ is bounded in $H$.
\end{itemize}
\end{definition}

The set of all {\sf nb}-bounded ({\sf bb}-bounded) homomorphisms
from a topological group $G$ to a topological group $H$ is denoted
by ${\sf Hom_{nb}}(G,H)$ (${\sf Hom_{bb}}(G,H)$). We write ${\sf
Hom}(G)$ instead of ${\sf Hom}(G,G)$. Here, we emphasize the group operation in  ${\sf Hom}(G,H)$ is pointwise, that is $(T+S)(x):=T(x)+S(x)$.

\smallskip
Now, assume $G$ is a topological group. The class of all ${\sf
nb}$-bounded homomorphisms on $G$ equipped with the topology of
uniform convergence on some neighborhood of $e_G$ is denoted by
${\sf Hom_{nb}}(G)$. Observe that a net $(S_{\alpha})$ of ${\sf
nb}$-bounded homomorphisms converges uniformly on a neighborhood $U$
of $e_G$ to a homomorphism $S$  if for each neighborhood $V$ of
$e_G$ there exists an $\alpha_0$ such that for each
$\alpha\geq\alpha_0$, $(S_{\alpha}-S)(U)\subseteq V$.

\smallskip
The class of all ${\sf bb}$-bounded homomorphisms on $G$ endowed
with the topology of uniform convergence on bounded sets is denoted
by ${\sf Hom_{bb}}(G)$. Note that a net $(S_{\alpha})$ of ${\sf
bb}$-bounded homomorphisms uniformly converges to a homomorphism $S$
on a bounded set $B\subseteq G$ if for each neighborhood $V$ of $e_G$
there is an $\alpha_0$ with $(S_{\alpha}-S)(B) \subseteq V$ for
each $\alpha\ge \alpha_0$.

\smallskip
The class of all continuous homomorphisms on $G$ equipped with the
topology of ${\sf c}$-convergence is denoted by ${\sf Hom_{c}}(G)$.
A net $(S_{\alpha})$ of continuous homomorphisms ${\sf c}$-converges
to a homomorphism $S$ if for each neighborhood $W$ of $e_G$, there
is a neighborhood $U$ of $e_G$ such that for every neighborhood $V$
of $e_G$ there exists an $\alpha_0$ with
$(S_{\alpha}-S)(U)\subseteq V+W$ for each $\alpha\geq\alpha_0$.

\smallskip
Note that ${\sf
Hom_{nb}}(G)$, ${\sf Hom_{c}}(G)$, and ${\sf Hom_{bb}}(G)$ form subgroups
of the group of all homomorphisms on $G$.

\section{topological lattice groups}
\begin{rem}
As opposed to topological vector spaces, in topological groups, not every singleton is bounded. In fact, scalar multiplication is a fruitful tool in this direction that we lack in topological groups; suppose $G$ is an abelian topological group and put $H=G\times {\Bbb Z}_2$. Then, $H$ is a topological group which contains unbounded singletons.  Nevertheless, in some cases such as many classical topological groups or connected topological groups, we do have this mild property. In this paper, we always assume that all topological groups have this mild property.
\end{rem}
\begin{exam}
Consider the additive group $\Bbb Z$ of integer numbers. It can be seen easily that with discrete topology, it is a locally solid topological group. Furthermore, it can be verified that it possesses Fatou property. But it is not a Riesz space, certainly.
\end{exam}
Recall that a homomorphism $T:G\to H$ is said to be order bounded if it maps order bounded sets into order bounded ones. The set of all order bounded homomorphisms from $G$ into $H$ is denoted by $\sf{Hom^{b}(G,H)}$. One may justify that under group operations of homomorphisms defined in \cite{KZ} and invoking \cite[Theorem 4.9]{H}, $\sf{Hom^{b}(G,H)}$ is a group.

\begin{lem}
Suppose $G$ is a Dedekind complete locally solid $\ell$-group with Fatou topology and ${\sf Hom^{b}_{n}}(G)$ is the group of all order bounded $nb$-bounded homomorphisms.
Then ${\sf Hom^{b}_{n}}(G)$ is an $\ell$-group.
\end{lem}
\begin{proof}
We need  to prove that for a homomorphism $T\in {\sf Hom^{b}_{n}}(G)$, $T^{+} \in {\sf Hom^{b}_{n}}(G)$. By \cite[Theorem 1]{Z2}, we have
\[T^{+}(x)=\sup \{T(u):  0\leq u\leq x\}.\]
Choose a neighborhood $U\subseteq G$ of the identity such that $T(U)$ is bounded. So, for arbitrary neighborhood $V$, there is $n\in \Bbb N$ with $T(U)\subseteq n V$. Therefore, for each $x\in U_{+}$, $T(x)\in n V$, so that $T^{+}(x)\in n V$ using solidness of $U$ and order closedness of $V$. Thus, we see that $T^{+}(U)$ is also bounded.

\end{proof}
\begin{thm}\label{600}
Suppose $G$ is a Dedekind complete locally solid $\ell$-group with Fatou topology. Then ${\sf Hom^{b}_{n}}(G)$ is locally solid with respect to the uniform convergence topology on some neighborhood at the identity.
\end{thm}
\begin{proof}
Let $T\in {\sf Hom^{b}_{n}}(G)$ and $x\in G_{+}$. By \cite[Theorem 1]{Z2}, we have
 \[T^{+}(x)=\sup\{T(u):  0\leq u\leq x\}.\]
   Now, suppose $(T_{\alpha})$ and $(S_{\alpha})$ are  nets of order bounded $nb$-bounded homomorphisms that $(T_{\alpha}-S_{\alpha})$ converges uniformly on some  neighborhood $U\subseteq G$ to zero. Choose arbitrary neighborhood $W\subseteq G$. Fix $x\in U_{+}$. Now, observe the following lattice inequality:
 \[\sup\{T_{\alpha}(u): 0\leq u\leq x\}-\sup\{S_{\alpha}(u): 0\leq u\leq x\}\]
 \[\le\sup\{(T_{\alpha}-S_{\alpha})(u): 0\leq u\leq x\}.\]
 There exists an $\alpha_0$ such that $(T_{\alpha}-S_{\alpha})(U)\subseteq W$ for each $\alpha\geq\alpha_0$. Therefore, using the order closedness of neighborhood $W$ and solidness of neighborhood $U$, we have
\[{T_{\alpha}}^{+}(x)-{S_{\alpha}}^{+}(x)\leq({T_{\alpha}-S_{\alpha}})^{+}(x)\in W.\]


Now, by considering \cite[Theorem 4.1]{H}, the proof would be complete.
\end{proof}
\begin{lem}
Suppose $G$ is a Dedekind complete locally solid $\ell$-group with Fatou topology and ${\sf Hom^{b}_{c}}(G)$ is the group of all order bounded continuous homomorphisms.
Then ${\sf Hom^{b}_{c}}(G)$ is an $\ell$-group.
\end{lem}
\begin{proof}
We need  to prove that for a homomorphism $T\in {\sf Hom^{b}_{c}}(G)$, $T^{+} \in {\sf Hom^{b}_{c}}(G)$. By \cite[Theorem 1]{Z2}, we have
\[T^{+}(x)=\sup \{T(u):  0\leq u\leq x\}.\]
Suppose $W\subseteq G$ is an arbitrary order closed neighborhood at the identity. There exists  a solid neighborhood $U$ with  $T(U)\subseteq V$. Therefore, for each $x\in U_{+}$, $T(x)\in  V$, so that $T^{+}(x)\in  V$ using solidness of $U$ and order closedness of $V$. Thus, we see that $T^{+}(U)\subseteq V$.

\end{proof}
\begin{thm}\label{600}
Suppose $G$ is a Dedekind complete locally solid $\ell$-group with Fatou topology. Then ${\sf Hom^{b}_{c}}(G)$ is locally solid with respect to the $c$-convergence topology.
\end{thm}
\begin{proof}
Let $T\in \sf{Hom^{b}_{c}}(G)$ and $x\in G_{+}$. By \cite[Theorem 1]{Z2}, we have
 \[T^{+}(x)=\sup\{T(u):  0\leq u\leq x\}.\]
  Suppose $(T_{\alpha})$ and $(S_{\alpha})$ are  nets of order bounded continuous homomorphisms that $(T_{\alpha}-S_{\alpha})$ $c$-converges to zero in ${\sf Hom^{b}_{c}}(X)$. Choose arbitrary neighborhood $W\subseteq G$. There is a neighborhood $U$ such that for every neighborhood $V$ there exists an $\alpha_0$ with $(T_{\alpha}-S_{\alpha})(U)\subseteq V+W$ for each $\alpha\geq\alpha_0$. Fix $x\in U_{+}$. Now, observe the following lattice inequality:
 \[\sup\{T_{\alpha}(u): 0\leq u\leq x\}-\sup\{S_{\alpha}(u): 0\leq u\leq x\}\]
 \[\le\sup\{(T_{\alpha}-S_{\alpha})(u): 0\leq u\leq x\}.\]
 Therefore, by considering the order closedness of neighborhoods $V$ and $W$ and also solidness of neighborhood $U$, we have
\[{T_{\alpha}}^{+}(x)-{S_{\alpha}}^{+}(x)\leq({T_{\alpha}-S_{\alpha}})^{+}(x)\in V+W.\]

Now, using \cite[Theorem 4.1]{H}, yields the desired result.


\end{proof}
\begin{lem}\label{700}
Suppose $G$ is a Dedekind complete locally solid $\ell$-group with Fatou topology and ${\sf Hom^{b}_{b}}(X)$ is the group of all order bounded $bb$-bounded homomorphisms.
Then ${\sf Hom^{b}_{b}}(G)$ is an $\ell$-group.
\end{lem}
\begin{proof}
It suffices to prove that for a homomorphism  $T\in {\sf Hom^{b}_{b}}(G)$, $T^{+} \in {\sf Hom^{b}_{b}}(G)$. By \cite[Theorem 1]{Z2}, we have
\[T^{+}(x)=\sup \{T(u):  0\leq u\leq x\}.\]
Suppose $V\subseteq G$ is an arbitrary neighborhood at the identity. Fix a bounded set $B\subseteq G$. Without loss of generality, we may assume $B$ is solid, otherwise, consider the solid hull of $B$ which is certainly bounded. There exists a positive integer $n$ with $T(B)\subseteq nV$. Therefore, for each $x\in B_{+}$, $T(x)\in  V$, so that $T^{+}(x)\in  V$ using solidness of $B$ and order closedness of $V$. Thus, we see that $T^{+}(B)\subseteq nV$.
\end{proof}
\begin{thm}\label{600}
Suppose $G$ is a Dedekind complete locally solid $\ell$-group with  Fatou topology. Then the lattice operations in ${\sf Hom^{b}_{b}}(G)$ are uniformly continuous with respect to the uniform convergence topology on bounded sets.
\end{thm}
\begin{proof}
Let $T\in \sf{Hom^{b}_{b}}(G)$ and $x\in G_{+}$. By \cite[Theorem 1]{Z2}, we have
 \[T^{+}(x)=\sup\{T(u):  0\leq u\leq x\}.\]
  Suppose $(T_{\alpha})$ and $(S_{\alpha})$ are  nets of order bounded $bb$-bounded homomorphisms that $(T_{\alpha}-S_{\alpha})$ converges uniformly to zero on bounded sets in ${\sf Hom^{b}_{b}}(X)$. Fix a bounded set $B\subseteq G$ which can be chosen solid  as in the proof of Lemma \ref{700}. Choose arbitrary neighborhood $W\subseteq G$. Fix $x\in B_{+}$. Now, observe the following lattice inequality:
 \[\sup\{T_{\alpha}(u): 0\leq u\leq x\}-\sup\{S_{\alpha}(u): 0\leq u\leq x\}\]
 \[\le\sup\{(T_{\alpha}-S_{\alpha})(u): 0\leq u\leq x\}.\]
 There exists an $\alpha_0$ such that $(T_{\alpha}-S_{\alpha})(B)\subseteq W$ for each $\alpha\geq\alpha_0$. Therefore, using the order closedness of neighborhood $W$ and solidness of subset $B$, we have
\[{T_{\alpha}}^{+}(x)-{S_{\alpha}}^{+}(x)\leq({T_{\alpha}-S_{\alpha}})^{+}(x)\in W.\]


Again, \cite[Theorem 4.1]{H} does the job.
This would complete our claim.
\end{proof}
\begin{rem}
As a side note, it can be noticed that if $G$ is a locally solid $\ell$-groups, then  ${\sf
Hom^{b}_{n}}(G)$, ${\sf Hom^{b}_{c}}(G)$, and ${\sf Hom^{b}_{b}}(G)$ are ideals in ${\sf Hom^{b}(G)}$.
\end{rem}
\subsection{unbounded topology}
In this part, we investigate unbounded topology on topological $\ell$-groups.

A net $(x_{\alpha})$ in a topological $\ell$-group $(G,\tau)$ is said to be {\bf unbounded $\tau$-convergent} to $x\in G$ ( in notation, $x_{\alpha}\xrightarrow{u\tau}x$) provided that $|x_{\alpha}-x|\wedge u\xrightarrow{\tau}0$ for each positive $u\in G$.
Note that for order bounded nets, $u\tau$-convergence and $\tau$-convergence agree. However, consider the additive group $c_0$ with topology $\tau$ induced by uniform norm and pointwise ordering; indeed, it is a topological $\ell$-group. Consider the sequence $(e_n)$ consists of the standard basis of $c_0$. Indeed, $e_n\xrightarrow{u\tau}0$ but not  in the $\tau$-topology.

Now, we show that this type of convergence is topological; more precisely, we prove that this kind of convergence on a locally solid $\ell$-group is again locally solid. For locally solid Riesz spaces, it is proved in \cite[Theorem 2.3]{T}.
We recall an elementary lemma which is a version of \cite[Lemma 1.4]{AB} in Riesz spaces.
\begin{lem}
If $x,x_1,x_2$ are positive elements in an $\ell$-group, then $x\wedge(x_1+x_2)\leq x\wedge x_1+x\wedge x_2$.
\end{lem}
\begin{thm}
Suppose $(G,\tau)$ is a locally solid $\ell$-group. Then $(G,u\tau)$ is again a locally solid $\ell$-group. If $\tau$ is Hausdorff, so is $u\tau$.
\end{thm}
\begin{proof}
Suppose $\{U_i\}_{i\in I}$ is a local basis of solid neighborhoods at identity for $G$. For each positive $u\in G$, put
\[U_{i,u}=\{x\in G, |x|\wedge u\in U_i\}.\]
We show that $\mathbf{B}:=\{U_{i,u}\}$ forms a basis for a locally solid topology on $G$ whose convergence is as the same as unbounded convergence. Note that since every $U_{i}$ is solid, we conclude that $U_{i,u}$ is also solid. In fact, we investigate properties of \cite[Theorem 3.5]{H}. For every index $i$, there is an $j$, such that $U_j+U_j\subseteq U_i$. Thus, for every positive element $u\in G$, one may verify $U_{j,u}+U_{j,u}\subseteq U_{i,u}$. It can be easily seen that each $U_{i,u}$ is symmetric.
For each $U_{i,u}$ and for each $y\in U_{i,u}$, there exists an index $j$ with $|y|\wedge u+U_{j}\subseteq U_i$. Now, observe that $y+U_{j,u}\subseteq U_{i,u}$.
For every $U\in \mathbf{B}$ and for every $x\in G$, we must show that there is a neighborhood $V\in \mathbf{B}$ such that $(V-x^{+})\wedge(V+x^{-})\subseteq U$.
Suppose $U=U_{i,u}$ for some $i$ and for some $u$. There exists an $j$ with $(U_{j}-x^{+})\vee(U_{j}+x^{-})\subseteq U_i$. We claim that $V:=U_{j,u}$ does the job.
Let $z\in V$ be fixed. By solidness of $U_{j,u}$, without loss of generality, we may assume that $z\geq 0$; otherwise consider $|z|$. We see that $z\wedge u\in U_{j}$. So,
\[0\leq (z+x^{-})\wedge u\leq z\wedge u+x^{-}\wedge u\leq z\wedge u+x^{-}.\]
By hypothesis, $z\wedge u+x^{-} \in U_i$ so that $(z+x^{-})\wedge u\in U_i$. Moreover, for each $w\in V$, we have
\[|(w-x^{+})\wedge (z+x^{-})|\leq |w-x^{+}|\wedge (z+x^{-})\leq z+x^{-}.\]
This implies that $(U_{j,u}-x^{+})\wedge(U_{j,u}+x^{-})\subseteq U_{i,u}$.

Finally, suppose $\tau$ is Hausdorff. We show that $u\tau$ is also Hausdorff. By \cite[Theorem 3.3]{H}, it is enough to prove that $\cap_{U\in \mathbf{B}} U=\{0\}$. Suppose $x\in U_{i,u}$ for all $i$ and for all $u\in G_{+}$. In particular, this means that $x\in U_{i,|x|}$ for all $i\in I$. Since $\tau$ is Hausdorff, we obtain the desired result.
\end{proof}
This point helps us to generalize some results dealing with unbounded convergence in locally solid Riesz spaces to locally solid $\ell$-groups; for example, a homomorphism $T$ between locally solid $\ell$-groups $(G,\tau)$ and $(H,\tau')$ is said to be {\bf unbounded Dunford-Pettis} ($u\tau$-Dunford-Pettis) if it maps every $\tau$-bounded $u\tau$-null net into $\tau'$-null nets. We finished this note with an extension of \cite[Proposition 4]{EGZ}, in this theme.
\begin{prop} Let $T\colon G\rightarrow H$ be a positive $u\tau$-Dunford-Pettis homomorphism between locally solid $\ell$-groups with $H$ Dedekind complete. Then the Kantorovich-like extension $S\colon G\rightarrow H$ defined via
\begin{center}
$S(y)=\sup \left\{ T(y\wedge y_{\alpha})\colon  (y_{\alpha})\subseteq G_{+},  y_{\alpha}\xrightarrow{u\tau} 0\right\}$
\end{center}
for every $y\in G_{+}$ is again $u\tau$-Dunford-Pettis.
\end{prop}
\begin{proof}
Suppose $y,z\in G_{+}$. Then
\begin{center}
$S(y+z)=\sup_{\beta} \{T((y+z)\wedge \gamma_{\beta})\}\le \sup_{\beta}\{T(y\wedge \gamma_{\beta})\}+\sup_{\beta}\{T(z\wedge \gamma_{\beta})\}\le S(y)+S(z)$,
\end{center}
in which, $(\gamma_{\beta})$ is a positive net that is $u\tau$-null. On the other hand,
\[T(y\wedge a_{\alpha})+T(z\wedge b_{\beta})=T(y\wedge a_{\alpha}+z\wedge b_{\beta})\leq T((y+z)\wedge (a_{\alpha}+b_{\beta}))\leq S(y+z),\]
provided that two positive nets $(a_{\alpha}),(b_\beta)$ are $u\tau$-null so that $S(y)+S(z)\leq S(y+z)$. Therefore, by \cite[Lemma 1]{Z2}, $S$ extends to a positive homomorphism. Denote by $S$ the extended homomorphism $S\colon G\rightarrow H.$

We show that $S$ is also $u\tau$-Dunford-Pettis. Suppose bounded net $(y_{\alpha})\subseteq X$ is $u\tau$-null. Therefore, we have
\[S(y_{\alpha})=\sup_{\beta}T(y_{\alpha}\wedge b_{\beta})\leq T(y_{\alpha})\rightarrow e_H,\]
in which $(b_{\beta})$ is a positive net in $G$ which is convergent to the identity in the $u\tau$-topology.
\end{proof}
\begin{rem}
Finally, it is worthwhile to mention that if a positive homomorphism $T$ is dominated by a $u\tau$-Dunford-Pettis homomorphism $S$, then $T$ is necessarily $u\tau$-Dunford-Pettis.
\end{rem}
\section{topological rings}
Now, we consider a version of \cite[Proposition 2.1]{Z} while scalar multiplication is absent. Recall that subset $B$ from a topological ring $X$ is called bounded if for each zero neighborhood $V\subseteq X$, there exists a zero neighborhood $U\subseteq X$ with $UB\subseteq V$ and $BU\subseteq V$. 
\begin{prop}\label{0}

Suppose $X$ is a topological ring with unity whose underlying topological group is connected. Then a set $B\subseteq X$ is bounded if and only if so is in the sense of a topological group.
\end{prop}
\begin{proof}
First, consider $X$ as a topological group and assume that $B\subseteq X$ is bounded. Furthermore, suppose $W$ is an arbitrary zero neighborhood. There is a zero neighborhood $V$ with $VV\subseteq W$. Find positive integer $n$ such that $B\subseteq nV$. Choose zero neighborhood $V_0$ with $nV_0\subseteq V$. Therefore, $V_0B\subseteq nV_0 V\subseteq VV\subseteq W$. Similarly, $BV_0\subseteq W$.

For the converse, consider $X$ as a topological ring and suppose $B\subseteq X$ is bounded. For an arbitrary zero neighborhood $W$ , there is a neighborhood $V$ with $VV\subseteq W$, $BV\subseteq W$ and $VB\subseteq W$. We claim there exists $n\in \Bbb N$ such that $B\subseteq n W$. Suppose on a contrary, for any $n\in \Bbb N$, $B\nsubseteq nW$. Since $X$ is connected, by \cite[Chapter III, Theorem 6]{Taq}, $X=\cup_{n=1}^{\infty}nV$. So, there are a sequence $(x_n)\subseteq B$ such that $x_n\notin nW$ and  an $m\in \Bbb N$ with $1\in mV$. So, $x_m\in mVB\subseteq mW$ a contradiction.

\end{proof}
Now, we recall some notes about bounded group homomorphisms between topological rings; for more expositions on this concept, see \cite{Z2}.
\begin{definition}\rm
Let $X$ and $Y$ be topological rings. A group homomorphism $T:X \to
Y$ is said to be
\begin{itemize}
\item[$(1)$] \emph{{\sf nr}-bounded} if there exists a
zero neighborhood $U\subseteq X$  such that $T(U)$ is bounded in $Y$;

\item[$(2)$] \emph{{\sf br}-bounded} if for every bounded set $B
\subseteq X$, $T(B)$ is bounded in $Y$.
\end{itemize}
\end{definition}

The set of all {\sf nr}-bounded ({\sf br}-bounded) homomorphisms
from a topological ring $X$ to a topological ring $Y$ is denoted
by ${\sf Hom_{nr}}(X,Y)$ (${\sf Hom_{br}}(X,Y)$). We write ${\sf
Hom}(X)$ instead of ${\sf Hom}(X,X)$.

\smallskip
Now, assume $X$ is a topological ring. The class of all ${\sf
nr}$-bounded group homomorphisms on $X$ equipped with the topology of
uniform convergence on some zero neighborhood is denoted by
${\sf Hom_{nr}}(X)$. Observe that a net $(S_{\alpha})$ of ${\sf
nr}$-bounded homomorphisms converges uniformly on a neighborhood $U$ to a homomorphism $S$  if for each neighborhood $V$  there exists an $\alpha_0$ such that for each
$\alpha\geq\alpha_0$, $(S_{\alpha}-S)(U)\subseteq V$.

\smallskip
The class of all ${\sf br}$-bounded group homomorphisms on $X$ endowed
with the topology of uniform convergence on bounded sets is denoted
by ${\sf Hom_{br}}(X)$. Note that a net $(S_{\alpha})$ of ${\sf
br}$-bounded homomorphisms uniformly converges to a homomorphism $S$
on a bounded set $B\subseteq X$ if for each zero neighborhood $V$
there is an $\alpha_0$ with $(S_{\alpha}-S)(B) \subseteq V$ for
each $\alpha\ge \alpha_0$.

\smallskip
The class of all continuous group homomorphisms on $X$ equipped with the
topology of ${\sf cr}$-convergence is denoted by ${\sf Hom_{cr}}(X)$.
A net $(S_{\alpha})$ of continuous homomorphisms ${\sf cr}$-converges
to a homomorphism $S$ if for each zero neighborhood $W$, there
is a neighborhood $U$ such that for every zero neighborhood $V$ there exists an $\alpha_0$ with
$(S_{\alpha}-S)(U)\subseteq VW$ for each $\alpha\geq\alpha_0$.

\smallskip
Note that ${\sf
Hom_{nr}}(X)$, ${\sf Hom_{br}}(X)$, and ${\sf Hom_{cr}}(X)$ form subrings
of the ring of all group homomorphisms on $X$, in which, the multiplication is given by function composition.

In contrast with the case of all bounded homomorphisms between topological groups ( considered in \cite{KZ}), there are no more relations between these classes of bounded group homomorphisms between topological rings; see \cite[Example 2.1, Example 2.2, Example 3.1]{Z} for some examples which illustrate the situation.
\begin{thm}
Suppose $X$ is a topological ring with unity. Then, ${\sf Hom_{cr}(X)}$ is complete if and only if so is $X$.
\end{thm}
\begin{proof}
Suppose $X$ is complete and $(T_{\alpha})$ is a Cauchy net of continuous group homomorphisms on $X$. Assume that $W$ is an arbitrary zero neighborhood. There is a zero neighborhood $U$ such that for any neighborhood $V$ we can choose an $\alpha_0$ with $(T_{\alpha}-T_{\beta})(U)\subseteq VW$ for each $\alpha\geq\alpha_0$ and $\beta\geq\alpha_0$. For any  fixed $x\in X$, find a positive integer $n$ such that $x\in nU$. Pick a zero neighborhood $V_0$ such that $V_0V_0\subseteq W$ and $nV_0\subseteq W$. Therefore, for sufficiently large $\alpha$ and $\beta$,  $(T_{\alpha}-T_{\beta})(x)\in nV_0 V_0\subseteq W$. Thus, $(T_{\alpha}(x))$ is a Cauchy net in $X$ so that convergent. Suppose $T_{\alpha}(x)\rightarrow \alpha_x\in X$. Define $T:X\to X$ via $T(x)=\alpha_x$. Since this convergence holds in ${\sf Hom_{cr}(X)}$, by \cite[Proposition 3.1]{Z}, $T$ is also $br$-bounded.

For the converse, assume that ${\sf Hom_{cr}(X)}$ is complete and $(x_{\alpha})$ is a Cauchy net in $X$. Suppose $W$ is an arbitrary zero neighborhood in $X$. Define $T_{\alpha}:X\to X$ via $T_{\alpha}(x)=x_{\alpha}x$. It can be verified that each $T_{\alpha}$ is a continuous group homomorphism. Furthermore, $(T_{\alpha})$ is a Cauchy net in ${\sf Hom_{cr}(X)}$; consider any neighborhood $U\subseteq W$, for any zero neighborhood $V$ choose index $\alpha_0$ such that $(x_{\alpha}-x_{\beta})\in V$ for each $\alpha\geq\alpha_0$ and $\beta\geq\alpha_0$ so that
\[(T_{\alpha}-T_{\beta})(U)=(x_{\alpha}-x_{\beta})U\subseteq VU\subseteq VW.\]
\end{proof}
\begin{thm}
Suppose $X$ is a topological ring with unity. Then, ${\sf Hom_{br}(X)}$ is complete if and only if so is $X$.
\end{thm}
\begin{proof}
Suppose $X$ is complete and $(T_{\alpha})$ is a Cauchy net of $br$-bounded group homomorphisms on $X$. Assume that $W$ is an arbitrary zero neighborhood in $X$ and fix $x\in X$ which is certainly bounded. There is an $\alpha_0$ such that $(T_{\alpha}-T_{\beta})(x)\in W$ for each $\alpha\geq\alpha_0$ and $\beta\geq\alpha_0$. Thus, $(T_{\alpha}(x))$ is a Cauchy net in $X$ so that convergent. Suppose $T_{\alpha}(x)\rightarrow \alpha_x\in X$. Define $T:X\to X$ via $T(x)=\alpha_x$. Since this convergence holds in ${\sf Hom_{br}(X)}$, by \cite[Proposition 2.2]{Z}, $T$ is also $br$-bounded.

For the converse, assume that ${\sf Hom_{br}(X)}$ is complete and $(x_{\alpha})$ is a Cauchy net in $X$. Suppose $W$ is an arbitrary zero neighborhood in $X$. Define $T_{\alpha}:X\to X$ via $T_{\alpha}(x)=x_{\alpha}x$. It can be verified that each $T_{\alpha}$ is a $br$-bounded group homomorphism. Furthermore, $(T_{\alpha})$ is a Cauchy net in ${\sf Hom_{br}(X)}$; for a fixed bounded set $B\subseteq X$, there is a zero neighborhood $V$ with $VB\subseteq W$. Choose index $\alpha_0$ such that $(x_{\alpha}-x_{\beta})\in V$ for each $\alpha\geq\alpha_0$ and $\beta\geq\alpha_0$ so that
\[(T_{\alpha}-T_{\beta})(B)=(x_{\alpha}-x_{\beta})B\subseteq VB\subseteq W.\]
\end{proof}
\begin{rem}
As opposed to the preceding cases, ${\sf Hom_{nr}(X)}$ does not behave well for completeness, in general. Consider \cite[Remark 2.2]{Z} for more details.
\end{rem}
Finally, we proceed with an affirmative answer for completeness of ${\sf Hom_{nr}(X)}$. First, we have the following fact.
\begin{prop}\label{00}
Suppose $X$ is a topological ring whose topological group is locally bounded. If $(T_{\alpha})$ is a net of $nr$-bounded group homomorphisms which is convergent uniformly on a zero neighborhood $U\subseteq X$ to a homomorphism $T$. Then $T$ is also $nr$-bounded.
\end{prop}
\begin{proof}
Assume that $W$ is an arbitrary neighborhood in $X$. There are a neighborhood $V$ with $V+V\subseteq W$ and a neighborhood $V_1$ such that $V_1V_1\subseteq V$. Find an $\alpha_0$ such that $(T_{\alpha}-T)(U)\subseteq V_1$ for each $\alpha\geq\alpha_0$. Fix an $\alpha\geq \alpha_0$. There is a neighborhood $U_1$ such that $T_{\alpha}(U_1)$ is bounded in $X$. Since $U$ is bounded, there is an $n\in\Bbb N$ with $U\subseteq nU_1$ so that $T_{\alpha}(U)\subseteq nT_{\alpha}(U_1)$. Observe that by hypothesis, $nT_{\alpha}(U_1)$ is bounded in $X$. So, there is a zero neighborhood $V_0$ such that $V_0\subseteq V_1$ and $V_0 nT_{\alpha}(U_1)\subseteq V$.
Therefore,
\[ V_0T(U)\subseteq V_0T_{\alpha}(U)+V_0V_1\subseteq V_0 n T_{\alpha}(U_1)+V_1V_1\subseteq V+V\subseteq W.\]
\end{proof}
Now, we consider a completeness characterization for  ${\sf Hom_{nr}(X)}$.
\begin{thm}
Suppose $X$ is a topological ring whose topological group is locally bounded. Then ${\sf Hom_{nr}(X)}$ is complete if and only if so is $X$.
\end{thm}
\begin{proof}
Suppose $X$ is complete and $(T_{\alpha})$ is a net which is uniformly Cauchy on some zero neighborhood $U\subseteq X$ of $nr$-bounded group homomorphisms. Assume that $W$ is an arbitrary zero neighborhood in $X$. There is an $\alpha_0$ such that $(T_{\alpha}-T_{\beta})(U)\subseteq  W$ for each $\alpha\geq\alpha_0$ and $\beta\geq\alpha_0$. Thus, for each $x\in U$, $(T_{\alpha}(x))$ is a Cauchy net in $X$ so that convergent. Fix any $x\in X$. There is a positive integer $n$ such that $x=ny$ for some $y\in U$. This means that $T_{\alpha}(x)$ is also Cauchy so that convergent. Suppose $T_{\alpha}(x)\rightarrow \alpha_x\in X$. Define $T:X\to X$ via $T(x)=\alpha_x$. Since this convergence holds in ${\sf Hom_{nr}(X)}$, by Proposition \ref{00}, $T$ is also $nr$-bounded.

For the converse, assume that ${\sf Hom_{nr}(X)}$ is complete and $(x_{\alpha})$ is a Cauchy net in $X$. Suppose $W$ is an arbitrary zero neighborhood in $X$. Define $T_{\alpha}:X\to X$ via $T_{\alpha}(x)=x_{\alpha}x$. It can be verified that each $T_{\alpha}$ is a $nr$-bounded group homomorphism. Furthermore, $(T_{\alpha})$ is a Cauchy net in ${\sf Hom_{nr}(X)}$; by assumption, there is a zero neighborhood $U$ which is bounded. So, there is a zero neighborhood $V$ with $VU\subseteq W$. Choose index $\alpha_0$ such that $(x_{\alpha}-x_{\beta})\in V$ for each $\alpha\geq\alpha_0$ and $\beta\geq\alpha_0$ so that
\[(T_{\alpha}-T_{\beta})(U)=(x_{\alpha}-x_{\beta})U\subseteq VU\subseteq W.\]
\end{proof}

\end{document}